\documentclass[12pt]{amsart} 
\makeatletter
\providecommand{\@LN}[2]{}
\makeatother
\calclayout

\usepackage{microtype}
\usepackage[english]{babel}
\usepackage[maxbibnames=5,minbibnames=4,backend=bibtex]{biblatex}
\usepackage{csquotes}
\usepackage{amssymb}
\usepackage{amsmath}
\usepackage{amsthm}
\newtheorem{thm}{Theorem}[section]
\newtheorem{lem}[thm]{Lemma}
\newtheorem{prop}[thm]{Proposition}

\theoremstyle{remark}
\newtheorem{rema}[thm]{Remark}

\newtheorem{defi}[thm]{Definition}

\usepackage{MnSymbol}
\usepackage{tikz}
\usepackage[all,cmtip]{xy}
\usepackage{todonotes}
\usepackage{stmaryrd}   
\title{Complete Lie algebroid actions and the integrability of Lie algebroids}
\author{Daniel \'Alvarez}\address{Departamento de Matem\'aticas, Facultad de Ciencias, Universidad Nacional Aut\'onoma de M\'exico, CP \texttt{04510}, M\'exico}
\email{verbum@ciencias.unam.mx}
\date{} 
\begin{document}


\begin{abstract} We give a new proof of the equivalence between the existence of a complete action of a Lie algebroid on a surjective submersion and its integrability. The main tools in our approach are double Lie groupoids and multiplicative foliations, our proof relies on a simple characterization of those vacant double Lie groupoids which induce a Lie groupoid structure on their orbit spaces.
\end{abstract}\maketitle
\tableofcontents 

\section{Introduction} Lie algebroids are basic objects in differential geometry, generalizing Lie algebras on one hand and tangent bundles on the other. Among many other geometric structures, Lie algebroids allow us to study Poisson manifolds, foliations and connections on principal bundles in a unified way \cite{macgen}. Generalizing the correspondence between Lie groups and Lie algebras, Lie algebroids can be seen as the infinitesimal objects corresponding to global structures which are known as Lie groupoids. Unlike finite-dimensional Lie algebras, not every Lie algebroid corresponds to a Lie groupoid. Those Lie algebroids which can be seen as the infinitesimal counterpart of a Lie groupoid are called integrable. It is of fundamental importance to recognize when a Lie algebroid is integrable \cite{crarui}, one of the main results in this respect is the equivalence between the existence of a complete action of a Lie algebroid on a surjective submersion and its integrability, see Theorem \ref{comsym} below which originally appeared formulated in terms of Poisson geometry \cite{craruipoi} since it settled a question posed in \cite{weican} about the existence of complete symplectic realizations. In this note we give a self-contained proof of this result and also we use a definition of completeness which seems to be weaker than the one that is used in \cite{craruipoi}. Our arguments do not depend on the construction of infinite-dimensional manifolds as in \cite{crarui,craruipoi} and are instead purely finite-dimensional, the key concept behind our proof is that of double Lie groupoid \cite{macdou}. Double Lie groupoids can be described as objects which encode global symmetries of Lie groupoids and so they are a useful tool for the study of quotients of Lie groupoids. In fact, the proof of Theorem \ref{comsym} can be roughly summarized as follows: a complete Lie algebroid action induces a multiplicative foliation on a submersion groupoid and this multiplicative foliation can be integrated to a double Lie groupoid. It turns out that the orbit space of this double Lie groupoid (with respect to one of its sides) is the desired integration of the Lie algebroid under consideration. The technical result that allows us to reach this conclusion describes the conditions that ensure the existence of a Lie groupoid structure on the orbit space of a vacant double Lie groupoid and it may be of independent interest, see Theorem \ref{vac3} below. 
\subsection*{Acknowledgments.} The author thanks CNPq for the financial support and H. Bursztyn for his continuous advice and support. The author also thanks the referee whose suggestions improved the readability of this work. 

\section{Preliminaries} \subsection{Lie groupoids and Lie algebroids} A {\em smooth groupoid} $G$ over a manifold $M$, denoted $G \rightrightarrows M$, is a groupoid object in the category of smooth manifolds (not necessarily Hausdorff) such that its source map is a submersion, see \cite{moeint,dufzun}. In this situation, $M$ is called the {\em base} of $G \rightrightarrows M$. The structure maps of a groupoid are its source, target, multiplication, unit map and inversion, denoted respectively $\mathtt{s},\mathtt{t},\mathtt{m},\mathtt{u} , \mathtt{i}$, we also denote $\mathtt{m}(a,b)$ by $ab$. A {\em Lie groupoid} is a smooth groupoid if its base and source-fibers are Hausdorff. When several groupoids are involved, we use a subindex $\mathtt{s}=\mathtt{s}_G$, $\mathtt{t}_G$, $\mathtt{m}_G$ to specify the groupoid under consideration. 

A Lie groupoid $G \rightrightarrows M$ is \emph{proper} if the map $(\mathtt{t},\mathtt{s}):G \rightarrow M \times M$ is proper. A Lie groupoid is \emph{source-simply-connected} if its source fibres are 1-connected. \emph{A left (right) Lie groupoid action} of a Lie groupoid $G \rightrightarrows M$ on a map $J:S \rightarrow M$ is a smooth map $a:G_\mathtt{s} \times_J S \rightarrow S$ (respectively, $a:S_J \times_\mathtt{t} G \rightarrow S$) which satisfies (1) $a(\mathtt{m} (g,h),x)=a(g,a(h,x))$ (respectively, $a(x,\mathtt{m} (g,h))=a(a(x,g),h)$) for all $g,h\in G$ and for all $x\in S$ for which $a$ and $\mathtt{m}$ are defined and (2) $a(\mathtt{u}(J(x)),x)= x$ for all $x\in M$. For the sake of brevity, we denote the fiber product $G_\mathtt{t} \times_J S $ by $G \times_M S$. In this situation, $G \times_M S$ becomes a Lie groupoid over $S$ with the projection being the source map, $a$ being the target map and the multiplication is defined by $(g,a(h,p))(h,p)=(gh,p)$. The Lie groupoid $G \times_M S\rightrightarrows S$ thus defined is called an {\em action groupoid}. Another construction which is relevant to us is the following: let $q:S \rightarrow M$ be a surjective submersion. Then the fiber product $S \times_M S$ is a Lie groupoid over $S$ with the source and target maps being the projections to $S$ and the multiplication given by $\mathtt{m}((x,y),(y,z))=(x,z)$. The Lie groupoid thus obtained is called the {\em submersion groupoid} associated to $q$.  

A {\em Lie algebroid} consists of the following data: (1) a vector bundle $A$ over a manifold $M$, (2) a bundle map $\mathtt{a}:A\rightarrow TM$ called the {\em anchor} and (3) a Lie algebra structure $[\,,\,]$ on $\Gamma(A)$ such that the Leibniz rule holds:
\[ [u,fv]=f[u,v]+\left(\mathcal{L}_{\mathtt{a}(u)}f\right)v, \] 
for all $u,v\in \Gamma(A)$ and $f\in C^\infty(M)$. The definition of Lie algebroid morphism is more involved, see \cite{higmacalg,vai}. The {\em tangent Lie algebroid} $\text{Lie} (G)$ of a Lie groupoid $G\rightrightarrows M$ is the vector bundle $\ker T\mathtt{s}|_{M}$ endowed with the restriction of $T\mathtt{t}$ to $A$ as the anchor and with the bracket determined by $[u,v]^r=[u^r,v^r]$ for all $u,v\in \Gamma (A)$, where $w^r\in \Gamma (G)$ is defined by $w^r(g)=T \mathtt{m}(w_{\mathtt{t}(g)},0_g)$ for all $w\in \Gamma (\ker T\mathtt{s}|_{M})$ and all $g\in G$. Also, we have that a Lie groupoid morphism induces a Lie algebroid morphism, thereby producing a functor from the category of Lie groupoids to the category of Lie algebroids which is called the {\em Lie functor} \cite{moeint}. 

A family of Lie algebroids which is relevant to us is given by the following construction. A Lie algebroid $A$ over $M$ acts on a map $J:S \rightarrow M$ if there is a Lie algebra morphism $\rho:\Gamma (A) \rightarrow \mathfrak{X}(S)$ such that $u_S:=\rho(u)$ is $J$-related to $\mathtt{a}(u)$ for all $u\in \Gamma (A)$. In this situation, the pullback vector bundle $J^*A$ inherits a Lie algebroid structure over $S$ called the {\em action Lie algebroid structure} \cite{macgen}. Since $\Gamma (J^*A)\cong C^\infty(S)\otimes_{C^\infty (M)}\Gamma (A) $, the Lie bracket on $\Gamma (J^*A)$ is determined by extending the bracket on $\Gamma (A) \hookrightarrow \Gamma (J^*A)$ using the Leibniz rule and the anchor $\widehat{\rho}:J^*A \rightarrow TS $ is defined by extending $\rho$ in a $C^\infty(S) $-linear fashion to $\Gamma (J^*A)$. The algebroid action $\rho$ of $A$ on $J$ is {\em complete} if $\rho(u)\in \mathfrak{X} (S)$ is a complete vector field whenever $\mathtt{a}(u)\in \mathfrak{X} (M)$ is complete.

\subsubsection{Double Lie groupoids and LA-algebroids} \begin{defi}[\cite{browmac,macdou}] A double topological groupoid is a groupoid object in the category of topological groupoids and is represented as follows:
\[ \xymatrix{ \mathcal{G}\ar@<-.5ex>[r] \ar@<.5ex>[r]\ar@<-.5ex>[d] \ar@<.5ex>[d]& H\ar@<-.5ex>[d] \ar@<.5ex>[d] \\ K\ar@<-.5ex>[r] \ar@<.5ex>[r] & S. }\]  
In the previous diagram, each of the sides denotes a groupoid structure and the structure maps of $\mathcal{G} $ over $K$ are groupoid morphisms with respect to the groupoids $\mathcal{G} \rightrightarrows H$ and $K \rightrightarrows S$. A {\em double Lie groupoid} is a double topological groupoid such that: (1) each of the side groupoids is a smooth groupoid, (2) $H$ and $K $ are Lie groupoids over $S$, and (3) the double source map $(\mathtt{s}^H,\mathtt{s}^K):  \mathcal{G} \rightarrow {H} \times_S {K}:=H_{\mathtt{s}_H} \times_{\mathtt{s}_K} K  $ is a surjective submersion (the superindices $\quad^H,\quad^K$ denote the groupoid structures $\mathcal{G} \rightrightarrows H$, $\mathcal{G} \rightrightarrows K$ respectively). \end{defi}
The infinitesimal counterpart of a double Lie groupoid is provided by the following concept. 
\begin{defi}[\cite{macdou}] An {\em LA-groupoid} is a Lie groupoid $\mathcal{B}  \rightrightarrows B$ such that: (1) $ \mathcal{B} $ and $B$ are Lie algebroids and all the structure maps are Lie algebroid morphisms over the structure maps of a base groupoid $G \rightrightarrows M$, (2) the map $\mathtt{s}'_{\mathcal{B} } :\mathcal{B}  \rightarrow \mathtt{s}_{G}^*B $ induced by $\mathtt{s}_{\mathcal{B} }$ is surjective. An LA-groupoid is \emph{vacant} if $\mathtt{s}'_{\mathcal{B}}$ is an isomorphism. \end{defi}
The most fundamental example of an LA-groupoid is the tangent bundle of a Lie groupoid. A \emph{multiplicative foliation} is an LA-subgroupoid of the tangent groupoid of a Lie groupoid \cite{groqua,muldir}.

\section{Multiplicative foliations and vacant double Lie groupoids}\label{sec:leavac} A Lie groupoid action on another Lie groupoid \cite[Definition 3.1]{higmacalg} induces a Lie groupoid structure on the quotient if the action is principal on the base, see \cite[Lemma 2.1]{intinfact}. In this section we generalize this observation thanks to the notion of vacant double Lie groupoid, see Theorem \ref{vac3}. Theorem \ref{vac3} will be the key ingredient in our proof that a Lie algebroid acting on a surjective submersion is integrable if such an action is complete, see Theorem \ref{comsym}.

\subsection{Vacant double Lie groupoids and their orbit spaces} The global counterpart of a vacant LA-groupoid is given by the following notion.
\begin{defi}[\cite{macdou}] A double Lie groupoid in which the double source map is a diffeomorphism is called {\em vacant}. \end{defi}
In the proof of the following theorem, we shall use the following equivalent description of a vacant double Lie groupoid. If $\mathcal{G} $ is a Lie groupoid over $S$ and $H,K$ are Lie subgroupoids such that the multiplication map ${H} \times_S K   \rightarrow \mathcal{G} $ is a diffeomorphism, then $\mathcal{G} $ is a vacant double Lie groupoid with sides ${H} $ and $K $. In fact, each element in $\mathcal{G} $ can be written as $\mathtt{m}(h,k)=\mathtt{m}(\overline{k},\overline{h})$ for some $h,\overline{h}\in {H} $ $k,\overline{k}\in K $. So we get two action groupoids $H \times_S K \rightrightarrows {K} $ and ${H} \times_S K \rightrightarrows H $ for the actions $(h,k)\mapsto {}^hk:=\overline{k}$, $(h,k)\mapsto h^k:=\overline{h}$ respectively. These structures make $\mathcal{G} $ with sides $H,K $ into a vacant double Lie groupoid over $S$ and all vacant double Lie groupoids are of this form, see \cite[Thm. 2.15]{macdou}. In one of the simplest situations, if $G$ and $G^*$ are 1-connected complete Poisson groups in duality, then we get a vacant double Lie groupoid on the double $G \times G^*$ with sides $G$ and $G^*$ over a point: the action groupoid structures over $G$ and over $G^*$ being the dressing actions and the group structure on $G \times G^*$ is the one that integrates the double $\mathfrak{g} \bowtie \mathfrak{g}^*$ of the tangent Lie bialgebra $(\mathfrak{g} ,\mathfrak{g}^*)$, see \cite{luphd,macdou}.
\begin{thm}\label{vac3} Let $\mathcal{G} $ be a vacant double Lie groupoid with sides ${H} ,K $ over $S$. If $H \rightrightarrows S$ is proper with trivial isotropy, then the orbit space $K/\mathcal{G} $ inherits a unique Lie groupoid structure over $S/H$ such that the projection $K \rightarrow K/\mathcal{G} $ is a Lie groupoid morphism. \end{thm} 

\begin{proof} Thanks to the description in the previous paragraph, we get that $\mathcal{G} \rightrightarrows K$ is isomorphic to the action groupoid $H \times_S K \rightrightarrows K$ associated to the action $(h,k)  \mapsto {}^hk$. The fact that $H \rightrightarrows S$ is proper with trivial isotropy implies that the orbit space $S/H$ inherits a unique manifold structure such that the projection $S \rightarrow S/H$ is a submersion, see \cite[Example 5.33]{moeint}\footnote{Notice that $H$ is automatically Hausdorff since the map $(\mathtt{t}_H,\mathtt{s}_H):H \rightarrow S \times S$ is injective and its image is Hausdorff.}. We shall see that the isotropy groups of $H \times_S K \rightrightarrows K$ are also trivial. Suppose that $(h,k)\in H \times_S K$ is such that ${}^hk=k$, this means that $hk={}^hkh^k=kh^k$ and so $\mathtt{s}(h)=\mathtt{t}(k)=\mathtt{t}(h)$. Then $h$ is a unit as desired. Since the $H$-action on $K$ is free, we get that $K/\mathcal{G} =K/H$ inherits a unique manifold structure such that the projection $K \rightarrow K/H$ is a submersion. In fact, if we take the pullback with $\mathtt{s}_K$ of a slice for the natural $H$-action on $S$, then we get a slice for the $H$-action on $K$, see the proof of \cite[Lemma 2.1]{intinfact}. 

We just have to check now that the groupoid structure on $K$ descends to $K/H$. For instance, let us check that $\mathtt{s}_K$ induces a well defined map $\overline{\mathtt{s} }: K/H \rightarrow S/H$. Take $(h,k)\in H \times_S K$, then $hk={}^hkh^k$ which means that $\mathtt{s}_K({}^hk)=\mathtt{t}(h^k)$. Since $h^k \in H$ satisfies that $\mathtt{s}_K(h^k)=\mathtt{s}_K(k) $, we get that $\mathtt{s}_K(k)$ and $\mathtt{s}_K({}^hk)$ lie in the same $H$-orbit as desired. It is somewhat less clear that the multiplication in $K/H$ is well defined. Take $k,k'\in K$ such that there exists $h\in H$ with the property that $\mathtt{s}_K(k)=\mathtt{t}_H(h)$ and $\mathtt{t}_K(k')=\mathtt{s}_H(h)$. So we define $\overline{\mathtt{m} }(H\cdot k,H\cdot k')=H\cdot \mathtt{m}_K(k,{}^hk')$. If $h'\in H$ is such that ${}^{h'}k$ is defined, let us notice that $\mathtt{s }_K({}^{h'}k)=\mathtt{t}_K({}^{(h'^k)h}k') $. We have to check that $\mathtt{m}_K(k,{}^hk')$ and $\mathtt{m}_K({}^{h'}k,{}^{(h'^k)h}k')$ lie in the same $H$-orbit. Indeed:
\[ {}^{h'}(k({}^hk'))h'^{k({}^hk')}= h'k({}^hk')={}^{h'}k(h'^k)({}^hk')={}^{h'}k({}^{(h'^k)h}k')h'^{k({}^hk')}  \] 
and hence we have that $({}^{h'}k)({}^{(h'^k)h}k')= {}^{h'}(k({}^hk'))$ as we wanted to show. Therefore, the groupoid structure on the quotient is well defined. \end{proof}  
\begin{rema} In the previous theorem, if  $\mathcal{G} \rightrightarrows K$ and $H \rightrightarrows S$ are the action groupoids associated to a Lie groupoid action  of $G \rightrightarrows M$ on $K \rightrightarrows S$ as in \cite[Definition 3.1]{higmacalg}, then we recover \cite[Lemma 2.1]{intinfact}. \end{rema}
Infinitesimally, the situation described by Theorem \ref{vac3} corresponds to a vacant multiplicative foliation given by the LA-groupoid $\text{Lie} (\mathcal{G} ) \rightrightarrows \text{Lie} (H)$ over $K \rightrightarrows S$. In general, the leaf space of a (vacant) multiplicative foliation does not inherit any groupoid structure compatible with the quotient map, see \cite{leaspa} for a general description of the necessary and sufficient conditions for that conclusion to hold. 	

\subsection{Monodromy groupoids of vacant multiplicative foliations} Now we are going to give a simple criterion that allows us to recognize when the monodromy groupoids of a vacant multiplicative foliation constitute a vacant double Lie groupoid. Let $K \rightrightarrows S$ be a Lie groupoid and let $\mathcal{B} \rightrightarrows B$ be multiplicative foliation on $H \rightrightarrows S$ which is vacant as an LA-groupoid. In this situation, we have the following result. 
\begin{prop}\label{intvac} Let $\mathcal{G} $ and $H$ be the monodromy groupoids of $\mathcal{B}$ and $B$ respectively. If the source map of $K$ restricts to a covering map between every compatible pair of $\mathcal{B} $ and $B$-orbits, then $\mathcal{G} $ is a vacant double Lie groupoid with sides $H$ and $K$. \end{prop} 
\begin{proof} Suppose that $\mathtt{s}_K$ restricted to every $\mathcal{B} $-orbit is a covering map and let $\mathtt{s}^H$, $\mathtt{t}^H$ be the Lie groupoid morphisms $\mathcal{G} \rightarrow H$ which integrate the source and target maps of $\mathcal{B} \rightrightarrows B$ respectively. Take $[a],[b]\in \mathcal{G} $ such that $\mathtt{s}^H(a)$ and $\mathtt{t}^H(b)$ are homotopic relative to their endpoints by means of $h$. Then we can lift $h$ to a homotopy from $a$ to $a'$ where $a'$ is pointwise composable with $b$, so we can define $\mathtt{m}^H( [a], [b])$ as the class of the path defined by $t\mapsto \mathtt{m}_K(a'(t),b(t))$. This multiplication over $H$ satisfies the interchange law with respect to $\mathcal{G} \rightrightarrows K$. It remains to check that the double Lie groupoid thus obtained is vacant. We claim that the double source map $(\mathtt{s}^H,\mathtt{s}^K):\mathcal{G} \rightarrow H \times_S K:=H_{\mathtt{s}_H} \times_{\mathtt{s}_K} K  $ is a local diffeomorphism. In fact, since $\mathtt{s}^H $ integrates the source map of $\mathcal{B} \rightrightarrows B$ which is fiberwise injective by the vacant assumption, we have that $ T\mathtt{s}^H $ restricted to $ \ker T \mathtt{s}^K $ is injective and then $T(\mathtt{s}^H,\mathtt{s}^K)$ has trivial kernel. On the other hand, $\mathcal{G} $ and $ H \times_S K$ have the same dimension so the claim follows. Finally, the homotopy lifiting property of the leaves implies that for every $k\in K$ and every $[a]\in H$ such that $a(0)=\mathtt{s}_K(k)$, there exists a unique path $\widetilde{a}$ tangent to $\mathcal{B} $ which starts at $k$ and satisfies $\mathtt{s}^H([\widetilde{a}])=[a]$. As a consequence, the double source map $(\mathtt{s}^H,\mathtt{s}^K):\mathcal{G} \rightarrow H \times_S K$ is a bijection and, therefore, it is a diffeomorphism. \end{proof}

The covering map condition in Proposition \ref{intvac} cannot be dropped: \cite[Example 2.4]{folpoigro} shows a vacant multiplicative foliation in which the associated monodromy groupoids do not admit any double Lie groupoid structure.

\section{Complete Lie algebroid actions and vacant double Lie groupoids} As an application of the arguments developed in the previous section, we recover the following classical result. 
\begin{thm}[\cite{craruipoi,conlin}]\label{comsym} A Lie algebroid $A \rightarrow M$ is integrable if there exists a complete action of $A$ on a surjective submersion $q:S \rightarrow M$ which induces an injective anchor $q^*A \rightarrow TS $ on the associated action Lie algebroid. Moreover, if the foliation induced by the $A$-action has no vanishing cycles, then $A$ is integrable by a Hausdorff Lie groupoid. \end{thm}
\begin{rema} R. L. Fernandes pointed out to the author that this result was proved in \cite{craruipoi} under an apparently stronger assumption on the action map $\rho: \Gamma (A) \rightarrow \mathfrak{X}(S) $: $\rho(u_t)\in \mathfrak{X}(S) $ should be complete for every time-dependent smooth family $u_t\in \Gamma (A)$ such that $\mathtt{a} (u_t)\in \mathfrak{X}(M) $ is complete. We do not know if this condition is equivalent to the one we use here. \end{rema}
Let us stress that in the following proof we shall only use monodromy groupoids of foliations so we will not need to consider Banach manifolds as in \cite{crarui,craruipoi}.
The main steps in our proof are the following:
\begin{enumerate} \item first of all, we consider the multiplicative foliation $\mathcal{B} \rightrightarrows B$ induced by the $A$-action on the submersion groupoid $S \times_M S \rightrightarrows S$ associated to $q$, 
\item we observe how the completeness assumption on the action implies that the monodromy groupoids associated to $\mathfrak{B} \rightrightarrows B$ constitute a vacant double Lie groupoid,
\item finally, we apply Theorem \ref{vac3} to the vacant double Lie groupoid thus obtained, since the submersion groupoid is proper with trivial isotropy. \end{enumerate}  
\begin{lem}\label{vac0} Let $\mathcal{B} \rightrightarrows B$ be a vacant multiplicative foliation on a Lie groupoid $K \rightrightarrows S$. Suppose that there is a subspace $\mathcal{W}\subset \Gamma (\mathcal{B} )\subset \mathfrak{X}   (K)$ consisting of complete multiplicative vector fields with the following property: for every $v\in B_p$ and every $p\in S$, there exists $\mathcal{X}\in \mathcal{W} $ such that $\mathcal{X}_p=v$. Then the $\mathcal{B} $-orbits are coverings of the corresponding $B$-orbits by means of $\mathtt{s} :K \rightarrow S$. \end{lem}
	\begin{proof} Take $p\in S$ and an ordered basis $\{e_i\}_{i=1\dots k}\subset B_p$. Let $\mathcal{X}_i\in \mathcal{W} $ be vector fields such that $\mathcal{X}_i(p)=e_i$ for all $i$. Define a map $\Phi:\mathbb{R}^k  \rightarrow S$ as $\Phi(t_1,\dots,t_k)=\Phi^{t_k}_{\mathcal{X}_k}\circ \dots \circ\Phi^{t_1}_{\mathcal{X}_1}(p)$, where $\Phi^t_{\mathcal{X}_i}$ is the flow of $\mathcal{X}_i$. Then $\Phi $ is an injective immersion on a neighborhood $V\subset \mathbb{R}^k $ of the origin whose image $U=\Phi(V)$ is embedded in $S$ and is an open set in the $B$-orbit $\mathcal{O} $ that passes through $p$. For every $x\in K$ such that $\mathtt{s}(x)= p$ we can define a neighborhood of $x$ in the $\mathcal{B}$-orbit that passes through $x$ as follows. Take $\widetilde{U}_{x} =\{\Phi^{t_k}_{\mathcal{X}_k}\circ \dots \circ\Phi^{t_1}_{\mathcal{X}_1}(x) :(t_1,\dots, t_k)\in V\}$. We claim that $\mathtt{s}|:{\widetilde{U}_{x}} \rightarrow U$ is a diffeomorphism. In fact, since $\mathcal{X}_i $ is multiplicative, it is $\mathtt{s}$-related to $\mathcal{X}_i|_S$. So we have that $\mathtt{s}(\Phi^t_{\mathcal{X}_i }(x))= \Phi^t_{\mathcal{X}_i }(\mathtt{s}(x))$ for all $x\in K$ and all $t\in \mathbb{R} $ and then it follows that $\mathtt{s}(\Phi^{t_k}_{\mathcal{X}_k}\circ \dots \circ\Phi^{t_1}_{\mathcal{X}_1}(x))=\Phi(t_1,\dots,t_k)$ as long as $\mathtt{s}(x)=p$. As a consequence, $\mathtt{s}$ establishes a bijection between $\widetilde{U}_x$ and $U$. On the other hand, since the foliation $\mathcal{B} $ is vacant and the $\mathcal{X}_i$ are multiplicative, the vector fields $\mathcal{X}_i$ are also pointwise independent over $\widetilde{U}_x$ and hence $\mathtt{s}|:{\widetilde{U}_{x}} \rightarrow U$ is indeed a diffeomorphism. Let $x,y \in K$ be such that $\mathtt{s}(x)=\mathtt{s}(y)=p$. Since each $\Phi^{t_1}_{\mathcal{X}_1}$ is an automorphism of $K$, we have that $\widetilde{U}_{x}\cap \widetilde{U}_{y}=\emptyset $ if $x\neq y$. Then $\mathtt{s}^{-1}(U)\cap \widetilde{\mathcal{O} } \cong \coprod_{x\in \mathtt{s}^{-1}(p)\cap \widetilde{\mathcal{O} }}   \widetilde{U}_{[a]}$, where $\widetilde{\mathcal{O} }$ is the $\mathcal{B}$-orbit that passes through $x$. Therefore, $\mathtt{s}|: \widetilde{\mathcal{O} } \rightarrow \mathcal{O} $ is a covering map. \end{proof}

\begin{proof}[Proof of Theorem \ref{comsym}] Take a complete action $\rho:\Gamma (A) \rightarrow \mathfrak{X}(S)$ of $A$ on $q$ which induces an injective anchor $\widehat{\rho}:q^*A \rightarrow TS $, where $q:S \rightarrow M$ is a surjective submersion. There is an induced action of $A$ on $q\circ \mathtt{s}$, where $\mathtt{s}$ is the source map of the submersion groupoid $K:=S \times_M S \rightrightarrows S$. In fact, simply define the action map $\Gamma (A) \rightarrow \mathfrak{X}  (K)$ as follows: $u \mapsto (\rho(u),\rho(u))$ for all $u\in \Gamma (A)$. This action induces a vacant multiplicative foliation $\mathcal{B} \rightrightarrows B$ over $S \times_M S \rightrightarrows S$, where $B:=q^*A$ is the action Lie algebroid associated to $\rho$. Let $\mathcal{G} $ be the monodromy groupoid of $\mathcal{B} $ and let ${H}  $ be the monodromy groupoid of $B$. In order to apply Proposition \ref{intvac} we have to see that the $\mathcal{B} $-orbits cover the corresponding ${B} $-orbits by means of $\mathtt{s}$. This follows from Lemma \ref{vac0} by putting 
\[ \mathcal{W}=\{(\rho(u),\rho(u))\in \mathfrak{X}(K)| \text{$\forall u\in \Gamma (A)$ of compact support} \}. \]

	Therefore, Proposition \ref{intvac} implies that $\mathcal{G} $ is a vacant double Lie groupoid with sides ${H} $ and $K$. Since $S \times_M S \rightrightarrows S$ is proper with trivial isotropy, Theorem \ref{vac3} implies that the orbit space $H/ \mathcal{G} $ of $\mathcal{G} \rightrightarrows H$ inherits a unique Lie groupoid structure over $M$ such that the projection map $H \rightarrow H/\mathcal{G} $ is a Lie groupoid morphism. Since we have that $q^* \text{Lie} (H/\mathcal{G} )\cong q^*A$, it follows that $\text{Lie} (H/ \mathcal{G} )\cong A$. If the foliation induced by the $A$-action has no vanishing cycles, then ${H} $ is Hausdorff \cite[Sec. 5.2]{moeint}. This implies that the orbit space $H/\mathcal{G}$ is a Hausdorff integration of $A$, since it coincides with the orbit space $H/(S \times_M S)$ for the associated lifted action of $S \times_M S \rightrightarrows S$ on $H$, see the proof of Theorem \ref{vac3}. \end{proof}

\printbibliography

@article{leaspa,
	Author = {M. Jotz},
	Journal = {Journal of Geometric Mechanics},
	Number = {3},
	Pages = {313-332},
	Title = {The leaf space of a multiplicative foliation},
	Volume = {4},
	Year = {2012}}

@article{muldir,
	Author = {C. Ortiz},
	Journal = {Pac. Journal of Mathematics},
	Number = {2},
	Pages = {329-365},
	Title = {Multiplicative Dirac structures},
	Volume = {266},
	Year = {2013}}

@article{intinfact,
	Author = {I. Moerdijk and J. Mrcun},
	Journal = {American Journal of Mathematics},
	Number = {3},
	Pages = {567-593},
	Title = {On integrability of infinitesimal actions},
	Volume = {124},
	Year = {2000}}

@article{folpoigro,
	Author = {D. \'Alvarez},
	Journal = {Comptes Rendus. Math{\'e}matique},
	Number = {2},
	Pages = {217-226},
	Title = {Leaves of stacky Lie algebroids},
	Volume = {358},
	Year = {2020}}

@book{weican,
	Author = {A. {Cannas da Silva} and A. {Weinstein}},
	Publisher = {American Mathematical Soc.},
	Title = {Geometric models for noncommutative algebras},
	Volume = {10},
	Year = {1999}}

@article{conlin,
	Author = {M. Crainic and R. L. Fernandes},
	Journal = {Annals of Mathematics},
	Number = {2},
	Pages = {1121-1139},
	Title = {A geometric approach to Conn's linearization theorem},
	Volume = {173},
	Year = {2011}}

@article{groqua,
	Author = {E. Hawkins},
	Journal = {Journal of Symplectic Geometry},
	Number = {1},
	Pages = {61-125},
	Title = {A groupoid approach to quantization},
	Volume = {6},
	Year = {2008}}

@article{higmacalg,
	Author = {P. J. {Higgins} and K. C. H. Mackenzie},
	Journal = {Journal of Algebra},
	Number = {1},
	Pages = {194-230},
	Publisher = {Elsevier},
	Title = {Algebraic constructions in the category of Lie algebroids},
	Volume = {129},
	Year = {1990}}

@article{craruipoi,
	Author = {M. Crainic and R. L. Fernandes},
	Journal = {Journal of Differential Geometry},
	Number = {1},
	Pages = {71-137},
	Publisher = {Lehigh University},
	Title = {Integrability of Poisson brackets},
	Volume = {66},
	Year = {2004}}

@article{browmac,
	Author = {R. {Brown} and K. C. H. Mackenzie},
	Journal = {Journal of pure and applied algebra},
	Number = {3},
	Pages = {237-272},
	Publisher = {Elsevier},
	Title = {Determination of a double Lie groupoid by its core diagram},
	Volume = {80},
	Year = {1992}}

@article{macdou,
	Author = {K. C. H. Mackenzie},
	Journal = {Advances in Mathematics},
	Number = {2},
	Pages = {180-239},
	Title = {Double Lie algebroids and second-order geometry, I},
	Volume = {94},
	Year = {1992}}

@phdthesis{luphd,
	Author = {J.-H. Lu},
	School = {University of California, Berkeley},
	Title = {Multiplicative and affine Poisson structures on Lie groups},
	Year = {1990}}

@book{dufzun,
	Author = {J. P. Dufour and N. T. Zung},
	Publisher = {Birkh{\"a}user},
	Title = {Poisson structures and their normal forms},
	Volume = {242},
	Year = {2006}}

@article{crarui,
	Author = {M. Crainic and R. L. Fernandes},
	Journal = {Annals of Mathematics},
	Pages = {575-620},
	Publisher = {JSTOR},
	Title = {Integrability of Lie brackets},
	Year = {2003}}

@book{macgen,
	Author = {K. C. H. Mackenzie},
	Publisher = {Cambridge University Press},
	Title = {General theory of Lie groupoids and Lie algebroids},
	Volume = {213},
	Year = {2005}}

@book{moeint,
	Author = {I. Moerdijk and J. Mrcun},
	Publisher = {Cambridge University Press},
	Title = {Introduction to foliations and Lie groupoids},
	Volume = {91},
	Year = {2003}}

@article{vai,
	Author = {A. Y. Vaintrob},
	Journal = {Russian Mathematical Surveys},
	Number = {2},
	Pages = {428-429},
	Title = {Lie algebroids and homological vector fields},
	Volume = {52},
	Year = {1997}}
\end{document}